\documentclass[12pt]{article}

\usepackage{latexsym,amssymb,amsmath}
\newcommand{\C}{\mathbb C}
\newcommand{\z}{\mathbb Z}
\newcommand{\q}{\mathbb Q}
\newtheorem{lem}{Lemma}[section]

\newtheorem{co}[lem]{Corollary}
\newtheorem{thm}[lem]{Theorem}
\newtheorem{prop}[lem]{Proposition}

\newtheorem{qu}[lem]{Question}
\newenvironment{proof}{\textbf{Proof.}}{\newline\hspace*{\fill}{$\Box$}\\}

\begin{document}
\title{Strictly ascending HNN extensions in soluble groups}
\author{J.\,O.\,Button\\
Selwyn College\\
University of Cambridge\\
Cambridge CB3 9DQ\\
U.K.\\
\texttt{jb128@dpmms.cam.ac.uk}}
\date{}
\maketitle
\begin{abstract}
We show that there exist finitely generated soluble groups which
are not LERF but which do not contain strictly ascending HNN
extensions of a cyclic group. This solves Problem 16.2 in the
Kourovka notebook. We further show that there is a finitely
presented soluble group which is not LERF but which does not
contain a strictly ascending HNN extension of a polycyclic group.
\end{abstract}
\section{Introduction}

A group is said to be extended residually finite or {\bf ERF} if every
subgroup is an intersection of finite index subgroups. This is a vast
strengthening of the property of residual finiteness. Any virtually
polycyclic group is ERF, as first proved by Mal'cev. Other groups
can be ERF; for instance see \cite{rrvpa} for recent results, but
not a single other example of a finitely generated ERF group is known.
This is Question 8 in \cite{med}, although it dates back to the paper
\cite{jw} where it is shown that a finitely generated virtually soluble
group which is ERF must be virtually polycyclic. This was reproved in
\cite{alpko} and \cite{rrvic}. Indeed the latter paper extends this
``ERF implies virtually polycyclic'' result to a wider class of finitely
generated groups which includes all finitely generated linear and
all finitely generated elementary amenable groups. To further illustrate
the problem, the class of ERF groups is closed under taking quotients
and subgroups, but no member can contain a non-abelian free group. Thus
a new finitely generated ERF group would not be elementary amenable
but nor could it contain a non-abelian free group, and even residually finite
examples of this are hard to come by.

Another property that is considerably stronger than residual
finiteness but weaker than ERF is that of being locally extended
residual finite or {\bf LERF}. This is when every finitely generated
subgroup is the intersection of finite index subgroups and is
sometimes called subgroup separable (although the phrase 
subgroup separable originally referred to ERF, as can be seen
in older papers). This property is
better suited to finitely generated groups and is useful not just
in group theory but in 3-manifold topology, because it has important
consequences when the group is the fundamental group of a compact
3-manifold and the finitely generated subgroup is that of a surface.
Not all compact 3-manifolds have LERF fundamental group but it is an
open question as to whether hyperbolic 3-manifolds do. However a
classical result of M.\,Hall Jnr. is that finitely generated free
groups are LERF and this was extended to closed surface groups by
P.\,Scott in \cite{sco}. It is clear that the LERF property is closed
under taking subgroups as well as finite index supergroups (for the
latter claim see \cite{rrvpa} Lemma 4.2, which establishes this
for ERF and the proof generalises immediately for LERF). However a
quotient of a LERF group need not be LERF (this follows immediately
for any property which holds for finitely generated free groups but
not for all finitely generated groups).

As for determining that a group $G$ is not LERF, a useful method is
due to Blass and Neumann in \cite{bln}. This says that if $G$ contains
a subgroup $H=\langle t,B\rangle$ which is a strictly ascending HNN
extension of a finitely generated group $B$ with stable letter $t$ then
$G$ is not LERF because in a finite quotient of $G$ we must have $B$ and
$tBt^{-1}<B$ going to conjugate subgroups, thus they are of the same order
and hence equal, so we cannot separate $B$ from $tBt^{-1}$. The strictly
ascending condition means that if $A_1$ and $A_2$ are the associated
subgroups of the HNN extension then one is equal to the base $B$ and the
other is strictly contained in $B$, so that conjugation by $t$ induces an
injective but not surjective endomorphism of $B$.

However we can ask if this always detects the absence of LERF in a finitely
generated group. In particular we can examine soluble groups $S$ and the
subgroups of $S$ which are HNN extensions. The advantage of solubility is
that any HNN extension contained in $S$ is ascending, meaning that at least one
of $A_1$ and $A_2$ is equal to $B$ (as otherwise $S$ would contain a
non-abelian free group). If $A_1=A_2=B$ then we have a semidirect product
$B\rtimes\z$, otherwise if $B$ is finitely generated we can conclude that
$S$ is not LERF.

In \cite{kou} it is asked if a finitely generated soluble group which is
not LERF must contain a strictly ascending HNN extension of a cyclic
group. We show that this is not the case and we show further that there
exists a finitely presented soluble group which is not LERF but which
does not contain a strictly ascending HNN extension of a finitely 
generated abelian group, or even a polycyclic group.

\section{Strictly ascending HNN extensions}

We summarise the facts we will need about ascending and strictly
ascending HNN extensions. If $B$ is a group with an isomorphism
$\theta$ to a subgroup $A$ of $B$ then the {\bf ascending HNN
extension} $G=\langle t,B\rangle$ is formed by adjoining to 
$G*\langle t\rangle$ the relations $tbt^{-1}=\theta(b)$ over all $b\in B$
(or just over a generating set for $B$). This gives rise to the
associated homomorphism $\chi$ of the HNN extension which is defined by
$\chi(t)=1$ and $\chi(B)=0$. If $A=B$ then ker$(\chi)$ is equal to $B$
and $G$ is the semidirect product $B\rtimes_\theta\z$, but if
$tBt^{-1}=A<B$ then ker$(\chi)$ is a strictly ascending union
$\cup_{i=0}^\infty t^{-i}Bt^i$ of subgroups which are all isomorphic to
$B$. This means that the kernel must be infinitely generated; indeed if
$C$ is any finitely generated subgroup of ker $(\chi)$ then it must be
contained in $t^{-i}Bt^i$ for some $i$, and so $C$ is conjugate to a
finitely generated subgroup of $B$.

Any $g\in G$ can be expressed in the form $g=t^{-k}bt^l$ for some $b\in B$
and $k,l\geq 0$, thus $\chi(g)=l-k$. Moreover if $B$ is soluble then so is
ker$(\chi)$ because solubility is a local condition, and also $G$ which is
equal to $\mbox{ker}(\chi)\rtimes_\alpha\z$ 
where $\alpha$ is the automorphism of
ker$(\chi)$ that is induced under conjugation by $t$. A strictly
ascending HNN extension of a cyclic group must imply that the cyclic
group is the integers $\z$ and it will be isomorphic to the soluble
Baumslag-Solitar group $BS(1,m)=\langle x,y|yxy^{-1}=x^m\rangle$ for
$m\neq 0,\pm 1$.

Now suppose that $BS(1,m)\leq G=\langle t,B\rangle$ where $G$ is itself
a strictly
ascending HNN extension. On trying to locate the elements $x,y$ in $G$ we see
by applying $\chi$ that $x\in\mbox{ker}(\chi)$ and so, on conjugating 
$\langle x,y\rangle$ by
an appropriate power of $t$, we have $x\in B$ and $y=t^{-k}bt^l$ for
$b\in B$ and $k,l\geq 0$. But we can further conjugate $BS(1,m)$ by $t^k$
to get an element $x=a$ of infinite order and $y=bt^n$ for $a,b\in B$ 
and $n=l-k$. Thus from the
Baumslag-Solitar relation we must have $\theta^n(a)=b^{-1}a^mb$ if $n\geq 0$
and $\theta^{-n}(b^{-1}a^mb)=a$ otherwise.    

\begin{thm} The soluble group $G=\langle t,\z\times\z\rangle$, which is
an ascending HNN extension formed by taking the endomorphism
$\theta(u)=u^5v^{-1},\theta(v)=u^2$ where $u,v$ is the standard generating
pair for $\z\times\z$, is not LERF but does not contain a strictly 
ascending HNN extension of a cyclic group.
\end{thm}
\begin{proof} The endomorphism $\theta$ of $\z\times\z$ gives rise to a
matrix with determinant 2, so $\theta$ 
is injective but not surjective. Thus $G$ is soluble, 
but not LERF by the Blass-Neumann result. Now suppose $G$ contains
$BS(1,m)$ for $m\neq 0,\pm 1$. By the comment before the theorem, we have
here that the base $B=\z\times\z$ and so we obtain 
$a\in B-\{0\}$ and $m\in\z-\{0,\pm 1\}$ such that $\theta^n(a)=a^m$ 
for $n>0$ (or $\theta^{-n}(a^m)=a$ for $n<0$). Writing this additively and
regarding $\theta$ as an invertible
linear map of $\mathbb R^2$, we are claiming (in both cases) that
$\theta^n$ has the eigenvalue $m$, so $\theta$ has an
eigenvalue $\lambda\in\C$ where $\lambda^n=m$.

Now the product of the two eigenvalues $\lambda,\mu$ of $\theta$ is 2 and
the sum is 5, so $\lambda$ and $\mu=2/\lambda$ satisfy $x^2-5x+2=0$ and
are algebraic integers. This means that $\mu^n$ is too for $n>0$ and if
$n<0$ then we are done because $\lambda$ being an algebraic
integer would imply that $\lambda^{-n}=1/m$ would be too. But we also
have $m\mu^n=2^n$. This means that $\mu^n$ is in $\q$ thus is an integer
dividing $2^n$. Now we have $1\leq |\mu|\leq 2$ and $|\lambda+\mu|=
|2/\mu+\mu|\leq 4$ but this contradicts $\lambda+\mu=5$.
\end{proof}

Problem 16.2 in \cite{kou} asks whether a finitely generated solvable
group is LERF if and only if it does not contain $BS(1,m)$ for
$m>1$. Thus we see from Theorem 2.1 that the answer is no. 
However this group contains
(and indeed is) a strictly ascending HNN extension of a finitely
generated abelian group, and we might generalise the question by asking
for a finitely generated non-LERF soluble group which does not contain
subgroups of this sort. To make further progress we have the following
proposition which is the main tool we will be using.
\begin{prop}
Suppose that $G$ is a group possessing a homomorphism onto an abelian 
group $A$
with kernel $K$ and $H=\langle t,B\rangle$ is a subgroup of $G$ which
is a strictly ascending HNN extension with base $B$ and stable letter
$t$. Then the subgroup $S=\langle t,B\cap K\rangle$ is also a strictly
ascending HNN extension.
\end{prop}
\begin{proof}
We have that $tBt^{-1}<B$. As $K$ is normal in $G$, we certainly have
$tCt^{-1}\leq C$ where $C=B\cap K$. Suppose this is actually equality.
We take any element $b\in B$ and consider $b_0=tbt^{-1}$. This is also in
$B$ and $\theta(b_0)=\theta(b)$ as $A$ is abelian. Thus $bb_0^{-1}\in C$
so by assumption we have $c_0\in C$ with $tc_0t^{-1}=bb_0^{-1}$. Then
$tc_0bt^{-1}=b$ and $c_0b\in B$, giving the contraction $tBt^{-1}=B$.
\end{proof}

\begin{thm}
There exists a finitely presented soluble group which is not LERF but 
which does not contain a strictly ascending HNN extension of a finitely 
generated abelian group.
\end{thm}
\begin{proof}
The particular group $G$ we work with is the Baumslag-Remeslennikov
group. This is a finitely presented soluble group which is residually
finite; indeed it is linear over $\mathbb R$. It is formed in the
following way: let $W$ be the wreath product $\z\wr\z$ which can be
thought of as the semidirect product $\z^\infty\rtimes_\alpha\z$,
where $\z^\infty$ is the free abelian group on countably many
generators $a_i$ for $i\in\z$ and $\alpha$ acts as a shift up by 1
so that the stable letter $s$ generating $\z$ satisfies $sa_is^{-1}=
a_{i+1}=\alpha(a_i)$. Now $W$ is finitely generated but not finitely
presented. However we consider the endomorphism $\theta$ of $W$
defined by $\theta(s)=s$ and $\theta(a_i)=a_ia_{i+1}$. This is injective
as any element of $W$ can be expressed uniquely as $as^i$ for $a\in\z^\infty$
and $i\in\z$. However it is not surjective: if we let the degree $d(a)$ of
a non-zero element $a\in\z^\infty$ be the difference in the indices of
the highest non-zero entry of $a$ and the lowest then $d(\theta(a))=
d(a)+1$, and so $a_i$ has no preimage under $\theta$ as $d(a_i)=1$.

We then use $\theta$ to form $G=\langle t,W\rangle$ which is the
strictly ascending HNN extension with base $W$ and stable letter $t$, so that
$twt^{-1}=\theta(w)$ for $w\in W$. Consequently any element of $G$ can be
written as $t^{-m}as^it^n$ for $a\in\z^\infty$, $i\in\z$ which is
uniquely defined and $m,n\geq 0$. 
Clearly $G$ is not LERF because $W$ is finitely generated. 
Let us suppose that $G$ contains $H=\langle\tau,B\rangle$
which is a strictly ascending HNN extension with finitely generated
abelian base $B$ and stable letter $\tau$. Now we have the homomorphism
from $G$ to $\z$ associated to the decomposition of $G$ as $\langle t,W
\rangle$ with kernel $K$. By Proposition 2.2 we can replace $B$ with
$B\cap K$ (which we will henceforth call $B$) because $B\cap K$ must also be
finitely generated abelian, and $\langle\tau,B\cap K\rangle$ becomes the
new $H$. But $K$ is an ascending union and $B$ is finitely generated
so is contained in $t^{-i}Wt^i$ for some $i\geq 0$. We can therefore
replace $H$ with $t^iHt^{-i}=\langle t^i\tau t^{-i},t^iBt^{-i}\rangle$    
which is still a strictly ascending HNN extension where the base 
$t^iBt^{-i}$ (which we now also rename $B$) lies in $W$.

However there is also a homomorphism from $G$ to $\z$ given by the
exponent sum of $s$ in an element $g$ of $G$, and another application
of Proposition 2.2 means we now think of the base $B$ of $H$ as lying in
$W$ with exponent sum of $s$ equal to 0, which means $B$ lies in $\z^\infty$.
Thus we must have some $g\in G$ with $gBg^{-1}<B<\z^\infty$. We let
$g=t^{-m}as^it^n$ as above. Remembering that $t$ commutes with $s$,
$\z^\infty$ is abelian and $t\z^\infty t^{-1}<\z^\infty$, we obtain
\[s^iBs^{-i}<t^{m-n}Bt^{-(m-n)}\mbox{ and }t^{n-m}Bt^{-(n-m)}<s^{-i}Bs^i.\]
We cannot have $n=m$ as $B$ is not strictly contained in $s^{-i}Bs^i$ for
any $i$. If $n>m$ then we use the second formula and take an element $b$
in $B$ which has maximum degree. This exists because $B$ is finitely
generated. Conjugating $b$ by positive powers of $t$ will increase the degree
but shifting $b$ will keep it constant. Therefore $t^{n-m}bt^{-(n-m)}$
cannot lie in $s^{-i}Bs^i$. On the other hand, if $n<m$ then we use the
first formula and take an element $b$ of $B$ with minimum degree. Again
one must exist unless $B=\{0\}$, in which case we do not have strict
containment. Because the degree of all non-zero elements of $B$ must
increase under conjugation by $t^{m-n}$, we see that $s^ibs^{-i}$ cannot
be contained in the right hand side. 
\end{proof}

Note that the use of Proposition 2.2 in the proof of Theorem 2.3 means that
the only point where we used the fact that the finitely 
generated base $B$ was abelian was to conclude
that $B\cap K$ was also finitely
generated. Thus Theorem 2.3 applies without change of proof to show that
$G$ does not contain a strictly ascending HNN extension where the base
is any finitely generated group all of whose subgroups are finitely 
generated. Now the soluble groups with this property are precisely the
polycyclic groups, so we have the following Corollary.
\begin{co}
The Baumslag-Remeslennikov finitely presented soluble group is not LERF
but does not contain a strictly ascending HNN extension of a polycyclic
group.
\end{co}

\section{Further comments}

We have seen that there exist finitely generated (and even finitely 
presented) soluble groups $G$ where the failure of $G$ to be LERF
cannot be witnessed by using the Blass-Neumann result to find a
strictly ascending HNN extension with a finitely generated base that is
cyclic, abelian or even polycyclic. But the possibility still remains
that this result will always show the absence of LERF because of the
existence of some strictly ascending HNN extension of an arbitrary
finitely generated subgroup (which of course
will necessarily be soluble).
Thus the following question remains.
\begin{qu}
If $G$ is a finitely generated (or finitely presented) soluble
group which is
not LERF then must it contain a strictly ascending HNN extension of some
finitely generated group?
\end{qu}
One may want to include $G$ being residually finite in the hypothesis
in case of counterexamples which might be so nasty as to be far from
being residually finite (and hence even further from being LERF) 
as well as being unable to contain strictly ascending HNN extensions.
It is remarked in \cite{decrn} after Proposition 3.19 that it would
be interesting to characterise LERF groups amongst finitely generated
soluble groups but this is open even when restricted to metabelian
groups. The Proposition itself shows that the (standard) wreath product
$A\wr B$ is LERF when $A$ is finitely generated abelian and $B=\z$. 
This was extended in \cite{alpa} to when both $A$ and $B$ are finitely
generated abelian. 

We remark though that the answer to the above question when extended
to arbitrary finitely presented groups is a definite no, even in a
class of finitely presented groups which is regarded as generally
well behaved. This is the class of fundamental groups of
3-manifolds, and if the 3-manifold is compact then the fundamental
group is finitely presented. Theorem 4.1 of \cite{memp} states that
the fundamental group of any (not necessarily compact) 3-manifold
cannot contain a strictly ascending HNN extension with finitely
generated base. However there are certainly compact
3-manifolds $M$ whose
fundamental group is not LERF: the first was given in \cite{brks},
giving a residually finite and coherent counterexample $\pi_1(M)$.  

We can give one class of soluble groups where the answer to Question 3.1
is yes: these are the constructible soluble groups first introduced
in \cite{bambi}. A group is
constructible if it has a subgroup of finite index which can be
built up as an HNN extension (or amalgamated free product) where the
base and associated subgroups (or the factors and the amalgamated
subgroup) have previously been built in this way. All constructible
groups are finitely presented, although not vice versa (for instance
the Baumslag-Remeslennikov example in Section 2). If a constructible
group is soluble then, because of the ubiquity of free groups in
HNN extensions and amalgamated free products, we can drop the
amalgamated free product construction without loss. As for HNN
extensions, we may assume that they are all ascending. Consequently
constructible groups that are soluble are especially well behaved:
they are all residually finite and even linear over $\q$,
virtually torsion free and of finite Pr\"{u}fer rank, for instance
see \cite{lenr} Subsection 11.2. Moreover it is
clear why the answer to Question 3.1 is straightforward here. Either
we never use a strictly ascending HNN extension in building our 
group, but then we are always taking finite extensions or
the semidirect product by $\z$ with a group
previously obtained, thus staying in the class of virtually polycyclic groups
(and polycyclic groups if the result is soluble) which are LERF. 
Otherwise we will use a strictly ascending HNN extension
for the first time which means that here the base is polycyclic, 
and this will be contained in our final group.

We remark that the finitely generated subgroups of (virtually)
soluble constructible
groups are precisely the finitely generated, residually finite 
(virtually) soluble
groups of finite Pr\"{u}fer rank. This is Theorem A of \cite{bambi}
without addition of the word virtually. Otherwise suppose that $G$ is
a finitely generated residually finite group of finite Pr\"ufer rank
with a normal index $m$ subgroup $N$ which is soluble, and $C$ is a
soluble constructible group containing $N$. Then, as $N$ and $C$ are
contained in $GL(n,\q)$ for some $n$, we have (by \cite{weh} Lemma 2.3
say) that $G$ is a subgroup of $GL(mn,\q)$ with $G/N$ isomorphic to a
group $P$ of $m\times m$ permutation matrices. Then $G$ embeds in the
subgroup $S$ of $GL(mn,\q)$ which is the extension by $P$ of the group
made up of $m$ diagonal blocks, each with an entry in $C$. But $S$ is
constructible as it has an index $m$ subgroup isomorphic to $C\times
\ldots \times C$ ($m$ times) and the direct product of constructible
groups is constructible by \cite{bambi} Proposition 2(b).

These groups occur in various places: for
instance the finitely generated residually finite groups with 
polynomial subgroup growth are precisely those which are virtually
soluble of finite Pr\"ufer rank, see \cite{ls}. However even if we have
a group where the absence of LERF can be detected because it contains
a strictly ascending HNN extension of a finitely generated group, this
need not be true of its subgroups. For instance
we can take $\pi_1(M)$ above and form the free (or
direct) product $P$ of $\pi_1(M)$ with a strictly ascending HNN
extension of a finitely generated group, whereupon both $P$ and $\pi_1(M)$
fail to be LERF but only $P$ contains such an HNN extension.

\end{document}